\documentclass[11pt,reqno]{amsart}
\usepackage{amsmath}
\usepackage{amsfonts}
\usepackage{amssymb,latexsym}
\usepackage{cite}
\usepackage[height=190mm,width=130mm]{geometry}

\theoremstyle{plain}
\newtheorem{theorem}{Theorem}
\newtheorem{lemma}{Lemma}
\newtheorem{corollary}{Corollary}

\theoremstyle{definition}
\newtheorem{definition}{Definition}
\theoremstyle{remark}
\newtheorem{remark}{Remark}

\numberwithin{equation}{section} 
\input{tcilatex}

\begin{document}
\title[On a new generalization of Fibonacci hybrid numbers]{On a new
generalization of Fibonacci hybrid numbers}
\author{Elif TAN}
\address{Department of Mathematics, Ankara University, Science Faculty,
06100 Tandogan, Ankara, Turkey}
\email{etan@ankara.edu.tr}
\author{N. Rosa Ait-Amrane}
\address{USTHB, Faculty of Mathematics, RECITS Laboratory, 16111
Bab-Ezzouar, Algiers, Algeria}
\email{rosamrane@gmail.com}

\begin{abstract}
The hybrid numbers were introduced by Ozdemir \cite{ozdemir} as a new
generalization of complex, dual, and hyperbolic numbers. A hybrid number is
defined by $k=a+bi+c\epsilon +dh$, where $a,b,c,d$ are real numbers and $%
i,\epsilon ,h$ are operators such that $i^{2}=-1,\epsilon ^{2}=0,h^{2}=1$
and $ih=-hi=\epsilon +i$. This work is intended as an attempt to introduce
the bi-periodic Horadam hybrid numbers which generalize the classical
Horadam hybrid numbers. We give the generating function, the Binet formula,
and some basic properties of these new hybrid numbers. Also, we investigate
some relationships between generalized bi-periodic Fibonacci hybrid numbers
and generalized bi-periodic\ Lucas hybrid numbers.
\end{abstract}

\maketitle

\section{Introduction}

Non-commutative algebras play important role and have broad applications in
many areas, such as mathematics and physics. Hence it is worth to study and
investigate the properties of some special types of non-commutative
algebras. The real quaternion algebra is the first non-commutative division
algebra to be discovered and defined by 
\begin{equation*}
\mathbb{H}=\{a+bi+cj+dk\mid i^{2}=j^{2}=k^{2}=-1,ij=-ji=k\}
\end{equation*}%
where $a,b,c,d\in 
\mathbb{R}
.$ For a survey on the properties of quaternions, we refer to Hamilton's
book \cite{hamilton}, and for some special type of quaternions see \cite%
{horadam, halici}.

A new non-commutative number system, the hybrid numbers, were introduced by
Ozdemir \cite{ozdemir} as a generalization of complex numbers, dual numbers,
and hyperbolic numbers which are having the form $a+be$ with $%
e^{2}=-1,e^{2}=0,$ and $e^{2}=1,$ respectively. The set of hybrid numbers
are defined as%
\begin{equation}
\mathbb{K}=\left\{ a+bi+c\epsilon +dh\mid i^{2}=-1,\epsilon
^{2}=0,h^{2}=1,ih=-hi=\epsilon +i\right\}  \label{h}
\end{equation}%
where $a,b,c,d\in 
\mathbb{R}
.$

The addition, substraction and multiplication of two hybrid numbers $%
k_{1}=a_{1}+b_{1}i+c_{1}\epsilon +d_{1}h$ and $k_{2}=a_{2}+b_{2}i+c_{2}%
\epsilon +d_{2}h$ are defined as%
\begin{equation*}
k_{1}\pm k_{2}=\left( a_{1}\pm a_{2}\right) +\left( b_{1}\pm b_{2}\right)
i+\left( c_{1}\pm c_{2}\right) \epsilon +\left( d_{1}\pm d_{2}\right) h,
\end{equation*}%
\begin{eqnarray*}
k_{1}k_{2} &=&a_{1}a_{2}-b_{1}b_{2}+d_{1}d_{2}+b_{1}c_{2}+c_{1}b_{2} \\
&&+\left( a_{1}b_{2}+b_{1}a_{2}+b_{1}d_{2}-d_{1}b_{2}\right) i \\
&&+\left(
a_{1}c_{2}+c_{1}a_{2}+b_{1}d_{2}-d_{1}b_{2}+d_{1}c_{2}-c_{1}d_{2}\right)
\epsilon \\
&&+\left( a_{1}d_{2}+d_{1}a_{2}+c_{1}b_{2}-b_{1}c_{2}\right) h.\text{ \ \ \
\ \ \ \ \ \ \ \ \ \ \ \ \ \ \ \ \ }
\end{eqnarray*}

The multiplication of a hybrid number $k=a+bi+c\epsilon +dh$ by the real
scalar $s$ is defined as%
\begin{equation*}
sk=sa+sbi+sc\epsilon +sdh,
\end{equation*}%
and the norm of a hybrid number $k$ is defined by%
\begin{equation*}
\left\Vert k\right\Vert :=\sqrt{\left\vert C\left( k\right) \right\vert },
\end{equation*}%
where $C\left( k\right) :=k\overline{k}$ is the character of the hybrid
number $k$ and $\overline{k}:=a-bi-c\epsilon -dh$ is the conjugate of $k.$
Ozdemir's paper \cite{ozdemir} serves as an excellent reference to the
algebraic and geometric properties of the hybrid numbers.

Recently, many studies have been devoted to hybrid numbers whose components
are taken from special integer sequences such as Fibonacci, Lucas, Pell,
Jacobsthal sequences, etc. In particular, Szynal-Liana \cite{szynal3}
introduced the Horadam hybrid numbers as 
\begin{equation}
\mathbb{K}_{W,n}=W_{n}+W_{n+1}i+W_{n+2}\epsilon +W_{n+3}h  \label{H}
\end{equation}%
where $\left\{ W_{n}\right\} $ is the Horadam sequence defined by $%
W_{n}=pW_{n-1}-qW_{n-2}$ with arbitrary initial values $W_{0},W_{1}.$ In 
\cite{szynal, szynal1, szynal2, szynal3}, the authors studied several
properties of special type of hybrid numbers. The basic properties of k-Pell
hybrid numbers were investigated by Catarino \cite{catarino}. Also, Morales 
\cite{morales} considered the $\left( p,q\right) $-Fibonacci and $\left(
p,q\right) $-Lucas hybrid numbers and gave several relations between them.
Recently, motivated by the Szynal-Liana's paper, Senturk et al. \cite%
{senturk} derived summation formulas, matrix representations, general
bilinear formula, Honsberger formula, etc. regarding to the Horadam hybrid
numbers.

This work has been intended as an attempt to introduce a new generalization
of Horadam hybrid numbers, called as, bi-periodic Horadam hybrid numbers.
The bi-periodic Horadam hybrid numbers generalize the best known hybrid
numbers in the literature, such as Horadam hybrid numbers, Fibonacci\&Lucas
hybrid numbers, $k$-Pell hybrid numbers, Pell\&Pell-Lucas hybrid numbers,
Jacobsthal\&Jacobsthal-Lucas hybrid numbers, etc. The components of the
bi-periodic Horadam hybrid numbers belong to the bi-periodic Horadam
sequence $\left\{ w_{n}\right\} $ which is defined by the recurrence relation%
\begin{equation}
w_{n}=\chi \left( n\right) w_{n-1}+cw_{n-2},\text{ }n\geq 2  \label{1}
\end{equation}%
where $\chi \left( n\right) =a$ if $n$ is even, $\chi \left( n\right) =b$ if 
$n$ is odd with arbitrary initial conditions $w_{0},w_{1}$ and nonzero real
numbers $a,b$ and $c.$ It is clear that if we take $a=b=p$ and $c=-q$, then
it reduces to the classical Horadam sequence in \cite{horadam1}. For the
details of the bi-periodic Horadam sequences see \cite{yayenie, edson,
bilgici, sahin}.

The outline of this paper is as follows: In the rest of this section, we
give some necessary definitions and mathematical preliminaries, which is
required. In Section 2, we introduce the bi-periodic Horadam hybrid numbers
and give the generating function, the Binet formula, matrix representation
and several basic properties of these hybrid numbers such as Vajda's
identity, Catalan's identity, Cassini's identity, summation and binomial sum
formulas. In Section 3, we give some relationships between the generalized
bi-periodic Fibonacci hybrid numbers and the generalized bi-periodic Lucas
hybrid numbers. The final section is devoted to the conclusions.

The Binet formula for the bi-periodic Horadam sequence $\left\{
w_{n}\right\} $ is%
\begin{equation}
w_{n}=\frac{a^{\xi \left( n+1\right) }}{\left( ab\right) ^{\left\lfloor 
\frac{n}{2}\right\rfloor }}\left( A\alpha ^{n}-B\beta ^{n}\right) ,
\label{2}
\end{equation}%
where%
\begin{equation}
A:=\frac{w_{1}-\frac{\beta }{a}w_{0}}{\alpha -\beta }\text{ and }B:=\frac{%
w_{1}-\frac{\alpha }{a}w_{0}}{\alpha -\beta }.  \label{3}
\end{equation}%
Here $\alpha $ and $\beta $ are the roots of the polynomial $x^{2}-abx-abc,$
that is, $\alpha =\frac{ab+\sqrt{a^{2}b^{2}+4abc}}{2}$ and $\beta =\frac{ab-%
\sqrt{a^{2}b^{2}+4abc}}{2}$, and $\xi \left( n\right) =n-2\left\lfloor \frac{%
n}{2}\right\rfloor $ is the parity function, i.e., $\xi \left( n\right) =0$
when $n$ is even and $\xi \left( n\right) =1$ when $n$ is odd. Let assume $%
a^{2}b^{2}+4abc>0.$ Also we have $\alpha +\beta =ab,$ $\Delta :=$ $\alpha
-\beta =\sqrt{a^{2}b^{2}+4abc}$ and $\alpha \beta =-abc.$ If we take the
initial conditions $0$ and $1,$ we get the Binet formula of the generalized
bi-periodic Fibonacci sequence $\left\{ u_{n}\right\} $ as 
\begin{equation}
u_{n}=\frac{a^{\xi \left( n+1\right) }}{\left( ab\right) ^{\left\lfloor 
\frac{n}{2}\right\rfloor }}\left( \frac{\alpha ^{n}-\beta ^{n}}{\alpha
-\beta }\right)  \label{u}
\end{equation}%
and by taking the initial conditions $2$ and $b,$ we get the Binet formula
of the generalized bi-periodic Lucas sequence $\left\{ v_{n}\right\} $ as%
\begin{equation}
v_{n}=\frac{a^{-\xi \left( n\right) }}{\left( ab\right) ^{\left\lfloor \frac{%
n}{2}\right\rfloor }}\left( \alpha ^{n}+\beta ^{n}\right) .  \label{v}
\end{equation}

The bi-periodic Horadam numbers for negative subscripts is defined as%
\begin{equation}
\left( -c\right) ^{n}w_{-n}=\left( \frac{b}{a}\right) ^{\xi \left( n\right)
}w_{0}u_{n+1}-w_{1}u_{n}.  \label{n}
\end{equation}

Also we have 
\begin{equation}
\alpha ^{m}=a^{-1}a^{\frac{m+\xi (m)}{2}}b^{\frac{m-\xi (m)}{2}}\alpha
u_{m}+ca^{\frac{m-\xi (m)}{2}}b^{\frac{m+\xi (m)}{2}}u_{m-1}  \label{*}
\end{equation}%
and%
\begin{equation}
\beta ^{m}=a^{-1}a^{\frac{m+\xi (m)}{2}}b^{\frac{m-\xi (m)}{2}}\beta
u_{m}+ca^{\frac{m-\xi (m)}{2}}b^{\frac{m+\xi (m)}{2}}u_{m-1}.  \label{**}
\end{equation}%
For details, see \cite{tan, tan1}.

\section{The bi-periodic Horadam hybrid numbers}

\begin{definition}
For $n\geq 0$, the bi-periodic Horadam hybrid number $\mathbb{K}_{w,n}$ is
defined by the recurrence relation 
\begin{equation*}
\mathbb{K}_{w,n}=w_{n}+w_{n+1}i+w_{n+2}\epsilon +w_{n+3}h
\end{equation*}%
where $w_{n}$ is the $n$-th bi-periodic Horadam number.
\end{definition}

From the definition of bi-periodic Horadam hybrid numbers, we have 
\begin{eqnarray*}
\mathbb{K}_{w,0} &=&w_{0}+w_{1}i+\left( aw_{1}+cw_{0}\right) \epsilon
+(\left( ab+c\right) w_{1}+bcw_{0})h, \\
\mathbb{K}_{w,1} &=&w_{1}+\left( aw_{1}+cw_{0}\right) i+(\left( ab+c\right)
w_{1}+bcw_{0})\epsilon \\
&&+(a\left( ab+2c\right) w_{1}+c\left( ab+c\right) w_{0})h.
\end{eqnarray*}

In the following table we state several number of hybrid numbers in terms of
the bi-periodic Horadam hybrid numbers $\mathbb{K}_{w,n}$ according to the
initial conditions $w_{0},w_{1}$ and the related coefficients $a,b,c.$

\begin{equation*}
\begin{array}{c}
\\ 
\begin{tabular}{|l|l|l|}
\hline
$\mathbf{K}_{w,n}$ & $\left( \mathbf{w}_{0}\mathbf{,w}_{1}\mathbf{;a,b,c}%
\right) $ & \textbf{bi-periodic Horadam hybrid numbers} \\ \hline
$\mathbb{K}_{u,n}$ & $\left( 0,1;a,b,c\right) $ & gen. bi-periodic Fibonacci
hybrid numbers \\ \hline
$\mathbb{K}_{v,n}$ & $\left( 2,b;a,b,c\right) $ & gen. bi-periodic Lucas
hybrid numbers \\ \hline
$\mathbb{K}_{W,n}$ & $\left( W_{0},W_{1};p,p,-q\right) $ & Horadam hybrid
numbers \cite{szynal3, senturk} \\ \hline
$\mathbb{K}_{U,n}$ & $\left( 0,1;p,p,q\right) $ & $\left( p,q\right) $%
-Fibonacci hybrid numbers \cite{morales} \\ \hline
$\mathbb{K}_{V,n}$ & $\left( 2,p;p,p,q\right) $ & $\left( p,q\right) $-Lucas
hybrid numbers \cite{morales} \\ \hline
$\mathbb{K}_{F,n}$ & $\left( 0,1;1,1,1\right) $ & Fibonacci hybrid numbers 
\cite{szynal} \\ \hline
$\mathbb{K}_{L,n}$ & $\left( 2,1;1,1,1\right) $ & Lucas hybrid numbers \cite%
{szynal3} \\ \hline
$\mathbb{K}_{P,n}$ & $\left( 0,1;2,2,1\right) $ & Pell hybrid numbers \cite%
{szynal1} \\ \hline
$\mathbb{K}_{Q,n}$ & $\left( 2,2;2,2,1\right) $ & Pell-Lucas hybrid numbers 
\cite{szynal1} \\ \hline
$\mathbb{K}_{kP,n}$ & $\left( 0,1;2,2,k\right) $ & k-Pell hybrid numbers 
\cite{catarino} \\ \hline
$\mathbb{K}_{J,n}$ & $\left( 0,1;1,1,2\right) $ & Jacobsthal hybrid numbers 
\cite{szynal2} \\ \hline
$\mathbb{K}_{j,n}$ & $\left( 2,1;1,1,2\right) $ & Jacobsthal-Lucas hybrid
numbers \cite{szynal2} \\ \hline
\end{tabular}
\\ 
\\ 
Table\text{ }1\label{Table1}:\text{Special cases of the sequence }\{\mathbb{K%
}_{w,n}\} \\ 
\\ 
\end{array}%
\end{equation*}%
The norm of the $n$-th bi-periodic Horadam hybrid number $\mathbb{K}_{w,n}$
is $\left\Vert \mathbb{K}_{w,n}\right\Vert :=\sqrt{\left\vert C\left( 
\mathbb{K}_{w,n}\right) \right\vert }.$ Here $C\left( \mathbb{K}%
_{w,n}\right) $ is the character of the $n$-th bi-periodic Horadam hybrid
number $\mathbb{K}_{w,n}$ and defined by%
\begin{equation}
C\left( \mathbb{K}_{w,n}\right) =\mathbb{K}_{w,n}\overline{\mathbb{K}_{w,n}}%
=w_{n}^{2}+\left( w_{n+1}-w_{n+2}\right) ^{2}-w_{n+2}^{2}-w_{n+3}^{2},
\label{4}
\end{equation}%
where $\overline{\mathbb{K}_{w,n}}$ $:=w_{n}-w_{n+1}i-w_{n+2}\epsilon
-w_{n+3}h$ is the conjugate of the bi-periodic Horadam hybrid number.

\begin{theorem}
The generating function for the bi-periodic Horadam hybrid sequence $\left\{ 
\mathbb{K}_{w,n}\right\} $ is%
\begin{equation*}
G\left( x\right) =\frac{\left( 1-\left( ab+c\right) x^{2}+bcx^{3}\right) 
\mathbb{K}_{w,0}+x\left( 1+ax-cx^{2}\right) \mathbb{K}_{w,1}}{1-\left(
ab+2c\right) x^{2}+c^{2}x^{4}}.
\end{equation*}
\end{theorem}

\begin{proof}
Let%
\begin{equation*}
G\left( x\right) =\sum_{n=0}^{\infty }\mathbb{K}_{w,n}x^{n}=\mathbb{K}_{w,0}+%
\mathbb{K}_{w,1}x+\mathbb{K}_{w,2}x^{2}+\cdots +\mathbb{K}_{w,n}x^{n}+\cdots
.
\end{equation*}%
Since the bi-periodic Horadam hybrid numbers satisfy the recurrence relation%
\begin{equation*}
\mathbb{K}_{w,n}=\left( ab+2c\right) \mathbb{K}_{w,n-2}-c^{2}\mathbb{K}%
_{w,n-4},\text{ }n\geq 4,
\end{equation*}
we get%
\begin{equation*}
\left( 1-\left( ab+2c\right) x^{2}+c^{2}x^{4}\right) G\left( x\right) \text{
\ \ \ \ \ \ \ \ \ \ \ \ \ \ \ \ \ \ \ \ \ \ \ \ \ \ \ \ \ \ \ \ \ \ \ \ \ \
\ \ \ \ \ \ \ \ \ \ \ \ \ \ \ \ \ \ \ \ \ \ \ \ \ \ \ \ \ \ \ \ \ \ \ \ \ \
\ \ \ \ \ \ \ \ \ \ \ \ \ \ \ \ \ \ \ \ \ \ \ \ \ \ \ \ \ }
\end{equation*}%
\begin{eqnarray*}
\text{\ \ } &=&\mathbb{K}_{w,0}+\mathbb{K}_{w,1}x+\left( \mathbb{K}%
_{w,2}-\left( ab+2c\right) \mathbb{K}_{w,0}\right) x^{2} \\
&&+\left( \mathbb{K}_{w,3}-\left( ab+2c\right) \mathbb{K}_{w,1}\right) x^{3}
\\
&&+\sum_{n=4}^{\infty }\left( \mathbb{K}_{w,n}-\left( ab+2c\right) \mathbb{K}%
_{w,n-2}+c^{2}\mathbb{K}_{w,n-4}\right) x^{n}\text{ \ \ \ \ \ \ \ \ \ \ \ \
\ \ \ \ \ \ \ \ \ \ \ \ \ \ \ \ \ \ \ \ \ \ \ \ \ \ \ \ \ \ \ \ \ \ \ \ \ \
\ \ \ \ \ \ \ \ \ \ \ \ \ \ \ \ \ \ \ \ \ \ \ \ }
\end{eqnarray*}%
\begin{eqnarray*}
&=&\mathbb{K}_{w,0}+\mathbb{K}_{w,1}x+\left( \left( a\mathbb{K}_{w,1}+c%
\mathbb{K}_{w,0}\right) -\left( ab+2c\right) \mathbb{K}_{w,0}\right) x^{2} \\
&&+\left( \left( \left( ab+c\right) \mathbb{K}_{w,1}+bc\mathbb{K}%
_{w,0}\right) -\left( ab+2c\right) \mathbb{K}_{w,1}\right) x^{3}.\text{ \ \
\ \ \ \ \ \ \ \ \ \ \ \ \ \ \ \ \ \ \ \ \ \ \ \ \ \ \ \ \ \ \ \ \ \ \ \ \ \
\ \ \ \ \ \ \ \ \ \ \ \ \ \ \ \ \ \ \ \ \ \ \ \ \ \ \ \ \ }
\end{eqnarray*}%
\ \ \ \ \ \ \ \ \ \ \ \ \ \ \ \ \ \ 
\end{proof}

Next, we state the Binet formula for the bi-periodic Horadam hybrid numbers
and so derive some well-known mathematical properties.

\begin{theorem}
The Binet formula for the bi-periodic Horadam hybrid numbers is%
\begin{equation*}
\mathbb{K}_{w,n}=\frac{a^{\xi \left( n+1\right) }}{\left( ab\right)
^{\left\lfloor \frac{n}{2}\right\rfloor }}\left( A\alpha _{\xi \left(
n\right) }\alpha ^{n}-B\beta _{\xi \left( n\right) }\beta ^{n}\right)
\end{equation*}%
where $\alpha _{\xi \left( n\right) }$ and $\beta _{\xi \left( n\right) }$
are defined as%
\begin{eqnarray*}
\alpha _{\xi \left( n\right) } &:&=1+\frac{1}{a}\left( \frac{a}{b}\right)
^{\xi \left( n\right) }\alpha i+\frac{1}{ab}\alpha ^{2}\epsilon +\frac{1}{%
a^{2}b}\left( \frac{a}{b}\right) ^{\xi \left( n\right) }\alpha ^{3}h, \\
\beta _{\xi \left( n\right) } &:&=1+\frac{1}{a}\left( \frac{a}{b}\right)
^{\xi \left( n\right) }\beta i+\frac{1}{ab}\beta ^{2}\epsilon +\frac{1}{%
a^{2}b}\left( \frac{a}{b}\right) ^{\xi \left( n\right) }\beta ^{3}h.
\end{eqnarray*}
\end{theorem}

\begin{proof}
By using the definition of the sequence $\{\mathbb{K}_{w,n}\}$ and the Binet
formula of $\left\{ w_{n}\right\} ,$ we obtain the desired result.
\end{proof}

\begin{remark}
If we take $a=b=p$ and $c=-q$, we obtain the Binet formula of the classical
Horadam hybrid numbers in \cite{szynal1}.
\end{remark}

\begin{lemma}
\label{l1}%
\begin{equation}
\alpha _{\xi \left( n\right) }\beta _{\xi \left( n\right) }=\left\{ 
\begin{array}{cc}
\mathbb{K}_{v,0}-\theta +\frac{\Delta }{a}c\left( \mathbb{K}_{u,0}-\eta
\right) , & \text{if }n\text{ is even} \\ 
\mathbb{K}_{\widehat{v},0}-\widehat{\theta }+\frac{\Delta }{b}c\left( 
\mathbb{K}_{\widehat{u},0}-\widehat{\eta }\right) , & \text{if }n\text{ is
odd}%
\end{array}%
\right.  \label{5}
\end{equation}%
\begin{equation}
\beta _{\xi \left( n\right) }\alpha _{\xi \left( n\right) }=\left\{ 
\begin{array}{cc}
\mathbb{K}_{v,0}-\theta -\frac{\Delta }{a}c\left( \mathbb{K}_{u,0}-\eta
\right) , & \text{if }n\text{ is even} \\ 
\mathbb{K}_{\widehat{v},0}-\widehat{\theta }-\frac{\Delta }{b}c\left( 
\mathbb{K}_{\widehat{u},0}-\widehat{\eta }\right) , & \text{if }n\text{ is
odd}%
\end{array}%
\right.  \label{6}
\end{equation}%
where%
\begin{eqnarray*}
\eta &:&=\left( 1-b\right) i+\left( a-b-c\right) \epsilon +\left(
1+ab+c\right) h, \\
\widehat{\eta } &:&=\left( 1-a\right) i+\left( b-a-c\right) \epsilon +\left(
1+ab+c\right) h, \\
\theta &:&=1-\frac{bc}{a}+bc+\frac{bc^{3}}{a}, \\
\widehat{\theta } &:&=1-\frac{ac}{b}+ac+\frac{ac^{3}}{b}.
\end{eqnarray*}%
and the sequences $\left\{ \mathbb{K}_{\widehat{u},0}\right\} $ and $\left\{ 
\mathbb{K}_{\widehat{v},0}\right\} $ are the auxiliary sequences that are
obtained from $\left\{ \mathbb{K}_{u,0}\right\} $ and $\left\{ \mathbb{K}%
_{v,0}\right\} $ just only switching $a\leftrightarrow b.$ That is, $%
\widehat{u_{n}}=\left( \frac{b}{a}\right) ^{\xi \left( n+1\right) }u_{n}$
and $\widehat{v_{n}}=\left( \frac{a}{b}\right) ^{\xi \left( n\right) }v_{n}.$
\end{lemma}

\begin{proof}
By using the definition of multiplication of two hybrid numbers, we have%
\begin{eqnarray*}
\alpha _{\xi \left( n\right) }\beta _{\xi \left( n\right) } &=&1+\frac{bc}{a}%
\left( \frac{a}{b}\right) ^{2\xi \left( n\right) }-\frac{bc^{3}}{a}\left( 
\frac{a}{b}\right) ^{2\xi \left( n\right) }-\left( \frac{a}{b}\right) ^{\xi
\left( n\right) }bc \\
&&+\left( b\left( \frac{a}{b}\right) ^{\xi \left( n\right) }+\left( \frac{a}{%
b}\right) ^{2\xi \left( n\right) }\frac{bc}{a}\Delta \right) i \\
&&+\left( \left( ab+2c\right) +\frac{bc}{a}\left( \frac{a}{b}\right) ^{2\xi
\left( n\right) }\Delta +\frac{c^{2}}{a}\left( \frac{a}{b}\right) ^{\xi
\left( n\right) }\Delta \right) \epsilon \\
&&+\left( \left( \frac{a}{b}\right) ^{\xi \left( n\right) }\left(
ab^{2}+3bc\right) -\frac{c}{a}\left( \frac{a}{b}\right) ^{\xi \left(
n\right) }\Delta \right) h\text{ \ \ \ \ \ \ \ \ \ \ \ \ \ \ \ \ \ \ \ \ \ \
\ \ \ \ \ \ \ \ }
\end{eqnarray*}%
\begin{eqnarray*}
&=&1+\left( \frac{a}{b}\right) ^{\xi \left( n\right) }bc\left( \frac{1}{a}%
\left( \frac{a}{b}\right) ^{\xi \left( n\right) }-\frac{c^{2}}{a}\left( 
\frac{a}{b}\right) ^{\xi \left( n\right) }-1\right) \\
&&+\left( \frac{a}{b}\right) ^{\xi \left( n\right) }\left( \left( \frac{b}{a}%
\right) ^{\xi \left( n\right) }2+bi+\left( \frac{b}{a}\right) ^{\xi \left(
n\right) }\left( ab+2c\right) \epsilon +\left( b\left( ab+3c\right) \right)
h\right) -2 \\
&&+\Delta \left( \frac{a}{b}\right) ^{\xi \left( n\right) }\frac{c}{a}\left(
\left( \frac{a}{b}\right) ^{\xi \left( n\right) }bi+\left( b\left( \frac{a}{b%
}\right) ^{\xi \left( n\right) }+c\right) \epsilon -h\right)
\end{eqnarray*}%
After some necessary simplifications, we get the result (\ref{5}).

Similarly, we can obtain $\beta _{\xi \left( n\right) }\alpha _{\xi \left(
n\right) }.$
\end{proof}

By using the Lemma \ref{l1}, we have%
\begin{eqnarray}
\alpha _{\xi \left( n\right) }\beta _{\xi \left( n\right) }+\beta _{\xi
\left( n\right) }\alpha _{\xi \left( n\right) } &=&\left\{ 
\begin{array}{cc}
2\left( \mathbb{K}_{v,0}-\theta \right) , & \text{ if }n\text{ is even} \\ 
2\left( \mathbb{K}_{\widehat{v},0}-\widehat{\theta }\right) , & \text{if }n%
\text{ is odd.}%
\end{array}%
\right.  \label{7} \\
\alpha _{\xi \left( n\right) }\beta _{\xi \left( n\right) }-\beta _{\xi
\left( n\right) }\alpha _{\xi \left( n\right) } &=&\left\{ 
\begin{array}{cc}
2\Delta \frac{c}{a}\left( \mathbb{K}_{u,0}-\eta \right) , & \text{ if }n%
\text{ is even} \\ 
2\Delta \frac{c}{b}\left( \mathbb{K}_{\widehat{u},0}-\widehat{\eta }\right) ,
& \text{if }n\text{ is odd.}%
\end{array}%
\right.  \label{8}
\end{eqnarray}

\begin{lemma}
\label{l2}%
\begin{equation}
\alpha _{\xi \left( n\right) }\alpha _{\xi \left( n\right) }=\left\{ 
\begin{array}{cc}
\mathbb{K}_{v,0}+\mu _{e}+\frac{\Delta }{a}\left( \mathbb{K}_{u,0}+\gamma
_{e}\right) , & \text{if }n\text{ is even} \\ 
\mathbb{K}_{\widehat{v},0}+\mu _{o}+\frac{\Delta }{b}\left( \mathbb{K}_{%
\widehat{u},0}+\gamma _{o}\right) , & \text{if }n\text{ is odd}%
\end{array}%
\right.  \label{9}
\end{equation}%
\begin{equation}
\beta _{\xi \left( n\right) }\beta _{\xi \left( n\right) }=\left\{ 
\begin{array}{cc}
\mathbb{K}_{v,0}+\mu _{e}-\frac{\Delta }{a}\left( \mathbb{K}_{u,0}+\gamma
_{e}\right) , & \text{if }n\text{ is even} \\ 
\mathbb{K}_{\widehat{v},0}+\mu _{o}-\frac{\Delta }{b}\left( \mathbb{K}_{%
\widehat{u},0}+\gamma _{o}\right) , & \text{if }n\text{ is odd}%
\end{array}%
\right.  \label{10}
\end{equation}%
where%
\begin{eqnarray*}
\mu _{e} &:&=-1+\frac{b}{a}c\left( u_{5}+2u_{2}-u_{1}\right) +b\gamma _{e} \\
\mu _{o} &:&=-1+\frac{a}{b}c\left( u_{5}+2\frac{b}{a}u_{2}-u_{1}\right)
+a\gamma _{o}
\end{eqnarray*}%
and%
\begin{eqnarray*}
\gamma _{e} &:&=\frac{1}{2}\left( \frac{b}{a}u_{6}+2u_{3}-\frac{b}{a}%
u_{2}\right) \\
\gamma _{o} &:&=\frac{1}{2}\left( u_{6}+2u_{3}-u_{2}\right) .
\end{eqnarray*}
\end{lemma}

\begin{proof}
By considering the relations%
\begin{equation*}
\alpha _{\xi \left( n\right) }\alpha _{\xi \left( n\right) }=2\alpha _{\xi
\left( n\right) }-C\left( \alpha _{\xi \left( n\right) }\right)
\end{equation*}%
and%
\begin{equation*}
\beta _{\xi \left( n\right) }\beta _{\xi \left( n\right) }=2\beta _{\xi
\left( n\right) }-C\left( \beta _{\xi \left( n\right) }\right) ,
\end{equation*}%
where $C\left( \alpha _{\xi \left( n\right) }\right) $ is the character of
the hybrid number $\alpha _{\xi \left( n\right) }$ and using the relations (%
\ref{*}) and (\ref{**}), we get the desired result.
\end{proof}

\begin{remark}
If we take $a=b=p$ and $c=q$, we obtain the analogous relations for $(p,q)$%
-Fibonacci hybrid numbers in \cite[Lemma 2.9]{morales}.
\end{remark}

\begin{theorem}
\textit{(Vajda's like identity)} For nonnegative integers $n,r,$ and $s,$ we
have%
\begin{eqnarray*}
&&\mathbb{K}_{w,n+2r}\mathbb{K}_{w,n+2s}-\mathbb{K}_{w,n}\mathbb{K}%
_{w,n+2\left( r+s\right) } \\
&=&\left\{ 
\begin{array}{cc}
\left( -c\right) ^{n}AB\Delta ^{2}u_{2r}\left( \left( \mathbb{K}%
_{v,0}-\theta \right) u_{2s}-c\left( \mathbb{K}_{u,0}-\eta \right)
v_{2s}\right) , & \text{if }n\text{ is even} \\ 
\left( -c\right) ^{n}AB\Delta ^{2}u_{2r}\left( \left( \mathbb{K}_{\widehat{v}%
,0}-\widehat{\theta }\right) \frac{b}{a}u_{2s}-c\left( \mathbb{K}_{\widehat{u%
},0}-\widehat{\eta }\right) v_{2s}\right) , & \text{if }n\text{ is odd.}%
\end{array}%
\right.
\end{eqnarray*}
\end{theorem}

\begin{proof}
From the Binet formula of the bi-periodic Horadam hybrid numbers, we get%
\begin{eqnarray*}
&&\mathbb{K}_{w,n+2r}\mathbb{K}_{w,n+2s}-\mathbb{K}_{w,n}\mathbb{K}%
_{w,n+2\left( r+s\right) } \\
&=&\frac{a^{\xi \left( n+2r+1\right) }}{\left( ab\right) ^{\left\lfloor 
\frac{n+2r}{2}\right\rfloor }}\left( A\alpha _{\xi \left( n\right) }\alpha
^{n+2r}-B\beta _{\xi \left( n\right) }\beta ^{n+2r}\right) \frac{a^{\xi
\left( n+2s+1\right) }}{\left( ab\right) ^{\left\lfloor \frac{n+2s}{2}%
\right\rfloor }}\left( A\alpha _{\xi \left( n\right) }\alpha ^{n+2s}-B\beta
_{\xi \left( n\right) }\beta ^{n+2s}\right) \\
&&-\frac{a^{\xi \left( n+1\right) }}{\left( ab\right) ^{\left\lfloor \frac{n%
}{2}\right\rfloor }}\left( A\alpha _{\xi \left( n\right) }\alpha ^{n}-B\beta
_{\xi \left( n\right) }\beta ^{n}\right) \frac{a^{\xi \left( n+2\left(
r+s\right) +1\right) }}{\left( ab\right) ^{\left\lfloor \frac{n+2\left(
r+s\right) }{2}\right\rfloor }}\left( A\alpha _{\xi \left( n\right) }\alpha
^{n+2\left( r+s\right) }-B\beta _{\xi \left( n\right) }\beta ^{n+2\left(
r+s\right) }\right)
\end{eqnarray*}%
\begin{eqnarray*}
&=&\frac{a^{2\xi \left( n+1\right) }}{\left( ab\right) ^{2\left\lfloor \frac{%
n}{2}\right\rfloor +r+s}}\left( -AB\alpha _{\xi \left( n\right) }\beta _{\xi
\left( n\right) }\alpha ^{n+2r}\beta ^{n+2s}-AB\beta _{\xi \left( n\right)
}\alpha _{\xi \left( n\right) }\alpha ^{n+2s}\beta ^{n+2r}\right. \\
&&\left. +AB\alpha _{\xi \left( n\right) }\beta _{\xi \left( n\right)
}\alpha ^{n}\beta ^{n+2\left( r+s\right) }+AB\beta _{\xi \left( n\right)
}\alpha _{\xi \left( n\right) }\alpha ^{n+2\left( r+s\right) }\beta
^{n}\right) \\
&=&\frac{a^{2\xi \left( n+1\right) }}{\left( ab\right) ^{2\left\lfloor \frac{%
n}{2}\right\rfloor +r+s}}AB\left( \alpha \beta \right) ^{n}\left( \alpha
_{\xi \left( n\right) }\beta _{\xi \left( n\right) }\beta ^{2s}\left( \beta
^{2r}-\alpha ^{2r}\right) +\beta _{\xi \left( n\right) }\alpha _{\xi \left(
n\right) }\alpha ^{2s}\left( \alpha ^{2r}-\beta ^{2r}\right) \right) \\
&=&\frac{a^{2\xi \left( n+1\right) }}{\left( ab\right) ^{2\left\lfloor \frac{%
n}{2}\right\rfloor +r+s}}AB\left( \alpha \beta \right) ^{n}\left( \alpha
^{2r}-\beta ^{2r}\right) \left( \beta _{\xi \left( n\right) }\alpha _{\xi
\left( n\right) }\alpha ^{2s}-\alpha _{\xi \left( n\right) }\beta _{\xi
\left( n\right) }\beta ^{2s}\right) .
\end{eqnarray*}

If $n$ is even, by considering the relations (\ref{5}) and (\ref{6}), we
obtain%
\begin{eqnarray*}
&&\mathbb{K}_{w,n+2r}\mathbb{K}_{w,n+2s}-\mathbb{K}_{w,n}\mathbb{K}%
_{w,n+2\left( r+s\right) } \\
&=&\frac{a^{2\xi \left( n+1\right) }}{\left( ab\right) ^{2\left\lfloor \frac{%
n}{2}\right\rfloor +r+s}}AB\left( \alpha \beta \right) ^{n}\left( \alpha
^{2r}-\beta ^{2r}\right) \left( \beta _{\xi \left( n\right) }\alpha _{\xi
\left( n\right) }\alpha ^{2s}-\alpha _{\xi \left( n\right) }\beta _{\xi
\left( n\right) }\beta ^{2s}\right) \\
&=&\frac{a^{2}\left( -c\right) ^{n}}{\left( ab\right) ^{r+s}}AB\left( \alpha
^{2r}-\beta ^{2r}\right) \left( \left( \mathbb{K}_{v,0}-\theta \right)
\left( \alpha ^{2s}-\beta ^{2s}\right) -\frac{\Delta }{a}c\left( \mathbb{K}%
_{u,0}-\eta \right) \left( \alpha ^{2s}+\beta ^{2s}\right) \right) \\
&=&\frac{a^{2}\left( -c\right) ^{n}}{\left( ab\right) ^{r+s}}AB\frac{\left(
ab\right) ^{r}}{a}\Delta u_{2r}\left( \left( \mathbb{K}_{v,0}-\theta \right)
\left( \frac{\left( ab\right) ^{s}}{a}u_{2s}\Delta \right) -\frac{\Delta }{a}%
c\left( \mathbb{K}_{u,0}-\eta \right) \left( \left( ab\right)
^{s}v_{2s}\right) \right) \\
&=&\left( -c\right) ^{n}ABu_{2r}\Delta ^{2}\left( \left( \mathbb{K}%
_{v,0}-\theta \right) u_{2s}-c\left( \mathbb{K}_{u,0}-\eta \right)
v_{2s}\right) .
\end{eqnarray*}%
Similarly, we obtain the desired result for odd $n$.
\end{proof}

\begin{corollary}
If we take $s=-r,$ we get the\textit{\ Catalan's like identity:}%
\begin{eqnarray*}
&&\mathbb{K}_{w,n+2r}\mathbb{K}_{w,n-2r}-\mathbb{K}_{w,n}^{2} \\
&=&\left\{ 
\begin{array}{cc}
\left( -1\right) ^{n+1}c^{n-2r}AB\Delta ^{2}u_{2r}\left( \left( \mathbb{K}%
_{v,0}-\theta \right) u_{2r}+c\left( \mathbb{K}_{u,0}-\eta \right)
v_{2r}\right) , & \text{if }n\text{ is even} \\ 
\left( -1\right) ^{n+1}c^{n-2r}AB\Delta ^{2}u_{2r}\left( \left( \mathbb{K}_{%
\widehat{v},0}-\widehat{\theta }\right) \frac{b}{a}u_{2r}+c\left( \mathbb{K}%
_{\widehat{u},0}-\widehat{\eta }\right) v_{2r}\right) , & \text{if }n\text{
is odd.}%
\end{array}%
\right.
\end{eqnarray*}
\end{corollary}

\begin{corollary}
If we take $s=-r$ and $r=1,$ we get the\textit{\ Cassini's like identity:}%
\begin{eqnarray*}
&&\mathbb{K}_{w,n+2}\mathbb{K}_{w,n-2}-\mathbb{K}_{w,n}^{2} \\
&=&\left\{ 
\begin{array}{cc}
\left( -1\right) ^{n+1}ac^{n-2}AB\Delta ^{2}\left( \left( \mathbb{K}%
_{v,0}-\theta \right) a+c\left( ab+2c\right) \left( \mathbb{K}_{u,0}-\eta
\right) \right) , & \text{if }n\text{ is even} \\ 
\left( -1\right) ^{n+1}ac^{n-2}AB\Delta ^{2}\left( \left( \mathbb{K}_{%
\widehat{v},0}-\widehat{\theta }\right) b+c\left( ab+2c\right) \left( 
\mathbb{K}_{\widehat{u},0}-\widehat{\eta }\right) \right) , & \text{if }n%
\text{ is odd.}%
\end{array}%
\right.
\end{eqnarray*}
\end{corollary}

Note that for even case, the Cassini's like identity can be stated as by
means of the following matrix identity:%
\begin{equation}
\left[ 
\begin{array}{cc}
\mathbb{K}_{w,2n+2} & \mathbb{K}_{w,2n} \\ 
\mathbb{K}_{w,2n} & \mathbb{K}_{w,2n-2}%
\end{array}%
\right] =\left[ 
\begin{array}{cc}
\mathbb{K}_{w,4} & \mathbb{K}_{w,2} \\ 
\mathbb{K}_{w,2} & \mathbb{K}_{w,0}%
\end{array}%
\right] \left[ 
\begin{array}{cc}
ab+2c & -c^{2} \\ 
1 & 0%
\end{array}%
\right] ^{n-1}.  \label{11}
\end{equation}%
By taking determinant from above to down below of both sides of the matrix
equality (\ref{11}), we get%
\begin{equation}
\mathbb{K}_{w,2n+2}\mathbb{K}_{w,2n-2}-\mathbb{K}_{w,2n}^{2}=c^{2n-2}\left( 
\mathbb{K}_{w,4}\mathbb{K}_{w,0}-\mathbb{K}_{w,2}^{2}\right) .  \label{12}
\end{equation}%
By taking determinant from down below to above of both sides of the matrix
equality (\ref{11}), we get%
\begin{equation}
\mathbb{K}_{w,2n-2}\mathbb{K}_{w,2n+2}-\mathbb{K}_{w,2n}^{2}=c^{2n-2}\left( 
\mathbb{K}_{w,0}\mathbb{K}_{w,4}-\mathbb{K}_{w,2}^{2}\right) .  \label{121}
\end{equation}

\begin{theorem}
For $n\geq 1,$ we have%
\begin{equation*}
\sum_{r=1}^{n}\mathbb{K}_{w,r}=\frac{c^{2}\left( \mathbb{K}_{w,n}+\mathbb{K}%
_{w,n-1}-\mathbb{K}_{w,0}-\mathbb{K}_{w,-1}\right) -\mathbb{K}_{w,n+2}-%
\mathbb{K}_{w,n+1}+\mathbb{K}_{w,2}+\mathbb{K}_{w,1}}{c^{2}-ab-2c+1}.
\end{equation*}
\end{theorem}

\begin{proof}
First note that by considering the formula in (\ref{n}), the bi-periodic
Horadam hybrid numbers for negative subscripts can be defined as%
\begin{equation*}
\mathbb{K}_{w,-n}=w_{-n}+w_{-n+1}i+w_{-n+2}\epsilon +w_{-n+3}h.
\end{equation*}

If $n$ is odd, we have\ 
\begin{equation*}
\sum_{r=1}^{n}\mathbb{K}_{w,r}=\sum_{r=1}^{\frac{n-1}{2}}\mathbb{K}%
_{w,2r}+\sum_{r=1}^{\frac{n+1}{2}}\mathbb{K}_{w,2r-1}\text{ \ \ \ \ \ \ \ \
\ \ \ \ \ \ \ \ \ \ \ \ \ \ \ \ \ \ \ \ \ \ \ \ \ \ \ \ \ \ \ \ \ \ \ \ \ \
\ \ \ \ \ \ \ \ \ \ \ \ \ \ \ \ \ \ \ \ \ \ \ \ \ \ \ \ \ \ \ \ \ \ \ \ \ \
\ \ \ \ \ \ \ \ \ \ \ \ \ }
\end{equation*}%
\begin{eqnarray*}
&=&\frac{a}{\left( ab\right) ^{\frac{n-1}{2}}}\left( \frac{A\alpha _{\xi
\left( n\right) }\alpha ^{n+1}-A\alpha _{\xi \left( n\right) }\alpha
^{2}\left( ab\right) ^{\frac{n-1}{2}}}{\alpha ^{2}-ab}+\frac{-B\beta _{\xi
\left( n\right) }\beta ^{n+1}+B\beta _{\xi \left( n\right) }\beta ^{2}\left(
ab\right) ^{\frac{n-1}{2}}}{\beta ^{2}-ab}\right) \\
&&+\frac{ab}{\left( ab\right) ^{\frac{n+1}{2}}}\left( \frac{A\alpha _{\xi
\left( n\right) }\alpha ^{n+2}-A\alpha _{\xi \left( n\right) }\alpha \left(
ab\right) ^{\frac{n+1}{2}}}{\alpha ^{2}-ab}+\frac{-B\beta _{\xi \left(
n\right) }\beta ^{n+2}+B\beta _{\xi \left( n\right) }\beta \left( ab\right)
^{\frac{n+1}{2}}}{\beta ^{2}-ab}\right) \text{ \ \ \ \ \ \ \ \ \ \ \ \ \ \ \
\ \ \ }
\end{eqnarray*}%
\begin{equation*}
=\sum_{r=1}^{\frac{n-1}{2}}\frac{a}{\left( ab\right) ^{r}}\left( A\alpha
_{\xi \left( n\right) }\alpha ^{2r}-B\beta _{\xi \left( n\right) }\beta
^{2r}\right) +\sum_{r=1}^{\frac{n-1}{2}+1}\frac{ab}{\left( ab\right) ^{r}}%
\left( A\alpha _{\xi \left( n\right) }\alpha ^{2r-1}-B\beta _{\xi \left(
n\right) }\beta ^{2r-1}\right) \text{ \ \ \ \ \ \ \ \ \ \ \ \ \ \ \ \ \ \ \
\ \ \ \ \ \ }
\end{equation*}%
\begin{eqnarray*}
&=&aA\alpha _{\xi \left( n\right) }\sum_{r=1}^{\frac{n-1}{2}}\left( \frac{%
\alpha ^{2}}{ab}\right) ^{r}-aB\beta _{\xi \left( n\right) }\sum_{r=1}^{%
\frac{n-1}{2}}\left( \frac{\beta ^{2}}{ab}\right) ^{r} \\
&&+\frac{ab}{\alpha }A\alpha _{\xi \left( n\right) }\sum_{r=1}^{\frac{n-1}{2}%
+1}\left( \frac{\alpha ^{2}}{ab}\right) ^{r}-\frac{ab}{\beta }B\beta _{\xi
\left( n\right) }\sum_{r=1}^{\frac{n-1}{2}+1}\left( \frac{\beta ^{2}}{ab}%
\right) ^{r}\text{ \ \ \ \ \ \ \ \ \ \ \ \ \ \ \ \ \ \ \ \ \ \ \ \ \ \ \ \ \
\ \ \ \ \ \ \ \ \ \ \ \ \ \ \ \ \ \ \ \ \ \ \ \ \ \ \ \ \ \ \ \ \ \ \ }
\end{eqnarray*}%
\begin{eqnarray*}
&=&aA\alpha _{\xi \left( n\right) }\left( \frac{\left( \frac{\alpha ^{2}}{ab}%
\right) ^{\frac{n-1}{2}+1}-\frac{\alpha ^{2}}{ab}}{\frac{\alpha ^{2}}{ab}-1}%
\right) -aB\beta _{\xi \left( n\right) }\left( \frac{\left( \frac{\beta ^{2}%
}{ab}\right) ^{\frac{n-1}{2}+1}-\frac{\beta ^{2}}{ab}}{\frac{\beta ^{2}}{ab}%
-1}\right) \\
&&+\frac{ab}{\alpha }A\alpha _{\xi \left( n\right) }\left( \frac{\left( 
\frac{\alpha ^{2}}{ab}\right) ^{\frac{n-1}{2}+2}-\frac{\alpha ^{2}}{ab}}{%
\frac{\alpha ^{2}}{ab}-1}\right) -\frac{ab}{\beta }B\beta _{\xi \left(
n\right) }\left( \frac{\left( \frac{\beta ^{2}}{ab}\right) ^{\frac{n-1}{2}%
+2}-\frac{\beta ^{2}}{ab}}{\frac{\beta ^{2}}{ab}-1}\right) \text{ \ \ \ \ \
\ \ \ \ \ \ \ \ \ \ \ \ \ \ \ \ \ \ \ \ \ \ \ \ \ \ \ }
\end{eqnarray*}%
\begin{eqnarray*}
&=&\frac{a}{\left( ab\right) ^{\frac{n-1}{2}}\left( \alpha ^{2}-ab\right)
\left( \beta ^{2}-ab\right) }\times \\
&&\left( \left( \alpha \beta \right) ^{2}\left( A\alpha _{\xi \left(
n\right) }\alpha ^{n-1}-B\beta _{\xi \left( n\right) }\beta ^{n-1}\right)
-ab\left( A\alpha _{\xi \left( n\right) }\alpha ^{n+1}-B\beta _{\xi \left(
n\right) }\beta ^{n+1}\right) \right. \\
&&\left. +\left( ab\right) ^{\frac{n-1}{2}}\left( -\left( \alpha \beta
\right) ^{2}\left( A\alpha _{\xi \left( n\right) }-B\beta _{\xi \left(
n\right) }\right) +ab\left( A\alpha _{\xi \left( n\right) }\alpha
^{2}-B\beta _{\xi \left( n\right) }\beta ^{2}\right) \right) \right) \text{\
\ \ \ \ \ \ \ \ \ \ \ \ \ \ \ \ \ \ \ \ \ \ \ \ \ \ \ \ \ \ \ \ \ \ \ \ \ \
\ \ \ }
\end{eqnarray*}%
\begin{equation*}
=\frac{c^{2}\left( \mathbb{K}_{w,n-1}-\mathbb{K}_{w,0}+\mathbb{K}_{w,n}-%
\mathbb{K}_{w,-1}\right) -\mathbb{K}_{w,n+1}-\mathbb{K}_{w,n+2}+\mathbb{K}%
_{w,2}+\mathbb{K}_{w,1}}{c^{2}-ab-2c+1}\text{ \ \ \ \ \ \ \ \ \ \ \ \ \ \ \
\ \ \ \ \ \ \ \ \ \ \ \ \ \ \ \ \ \ \ \ \ \ \ \ \ }
\end{equation*}

If $n$ is even, we have%
\begin{equation*}
\sum_{r=1}^{n}\mathbb{K}_{w,r}=\sum_{r=1}^{\frac{n}{2}}\mathbb{K}%
_{w,2r}+\sum_{r=1}^{\frac{n}{2}}\mathbb{K}_{w,2r-1}.
\end{equation*}%
In a similar manner, we get the desired result.
\end{proof}

\begin{theorem}
For nonnegative integers $n$ and $r$, we have%
\begin{equation*}
\left( i\right) \sum_{i=0}^{n}\binom{n}{i}\left( -c\right) ^{n-i}\mathbb{K}%
_{w,2i+r}=\left( ab\right) ^{\frac{n}{2}}\left( \frac{a}{b}\right) ^{\frac{%
\xi \left( n+r\right) -\xi \left( r\right) }{2}}\mathbb{K}_{w,n+r}.\text{\ \
\ \ \ \ \ \ \ \ \ \ \ \ \ \ \ \ \ \ \ \ \ \ \ \ \ \ \ \ \ \ \ \ \ \ \ \ \ \
\ \ \ \ \ \ \ \ \ \ \ \ \ \ \ \ \ \ \ \ \ \ \ }
\end{equation*}%
\begin{equation*}
\left( ii\right) \sum_{i=0}^{n}\binom{n}{i}c^{n-i}\left( ab\right) ^{\frac{i%
}{2}}\left( \frac{a}{b}\right) ^{\frac{\xi \left( i+r\right) -\xi \left(
r\right) }{2}}\mathbb{K}_{w,i+r}=\mathbb{K}_{w,2n+r}.\text{ \ \ \ \ \ \ \ \
\ \ \ \ \ \ \ \ \ \ \ \ \ \ \ \ \ \ \ \ \ \ \ \ \ \ \ \ \ \ \ \ \ \ \ \ \ \
\ \ \ \ \ \ \ \ \ \ \ \ \ \ \ \ \ \ \ \ \ }
\end{equation*}
\end{theorem}

\begin{proof}
$\left( i\right) $ From the Binet formula of the bi-periodic Horadam hybrid
numbers, we get 
\begin{equation*}
\sum_{i=0}^{n}\binom{n}{i}\left( -c\right) ^{n-i}\mathbb{K}_{w,2i+r}\text{ \
\ \ \ \ \ \ \ \ \ \ \ \ \ \ \ \ \ \ \ \ \ \ \ \ \ \ \ \ \ \ \ \ \ \ \ \ \ \
\ \ \ \ \ \ \ \ \ \ \ \ \ \ \ \ \ \ \ \ \ \ \ \ \ \ \ \ \ \ \ \ \ \ \ \ \ \
\ \ \ \ \ \ \ \ \ \ \ \ \ \ \ \ \ \ \ \ \ \ \ \ \ \ \ \ \ \ \ \ \ \ \ \ \ \ }
\end{equation*}%
\begin{equation*}
=\ \sum_{i=0}^{n}\binom{n}{i}\left( -c\right) ^{n-i}\frac{a^{\xi \left(
2i+r+1\right) }}{\left( ab\right) ^{\left\lfloor \frac{2i+r}{2}\right\rfloor
}}\left( A\alpha _{\xi \left( n\right) }\alpha ^{2i+r}-B\beta _{\xi \left(
n\right) }\beta ^{2i+r}\right) \text{ \ \ \ \ \ \ \ \ \ \ \ \ \ \ \ \ \ \ \
\ \ \ \ \ \ \ \ \ \ \ \ \ \ \ \ \ \ \ \ \ \ \ \ \ \ \ \ \ \ \ \ \ \ \ \ \ \ }
\end{equation*}%
\begin{equation*}
=\frac{a^{\xi \left( r+1\right) }}{\left( ab\right) ^{\left\lfloor \frac{r}{2%
}\right\rfloor }}A\alpha _{\xi \left( n\right) }\alpha ^{r}\sum_{i=0}^{n}%
\binom{n}{i}\left( -c\right) ^{n-i}\left( \frac{\alpha ^{2}}{ab}\right) ^{i}-%
\frac{a^{\xi \left( r+1\right) }}{\left( ab\right) ^{\left\lfloor \frac{r}{2}%
\right\rfloor }}B\beta _{\xi \left( n\right) }\beta ^{r}\sum_{i=0}^{n}\binom{%
n}{i}\left( -c\right) ^{n-i}\left( \frac{\beta ^{2}}{ab}\right) ^{i}\text{ \ 
}
\end{equation*}%
\begin{equation*}
=\frac{a^{\xi \left( r+1\right) }}{\left( ab\right) ^{\left\lfloor \frac{r}{2%
}\right\rfloor }}A\alpha _{\xi \left( n\right) }\alpha ^{r}\left( \frac{%
\alpha ^{2}}{ab}-c\right) ^{n}-\frac{a^{\xi \left( r+1\right) }}{\left(
ab\right) ^{\left\lfloor \frac{r}{2}\right\rfloor }}B\beta _{\xi \left(
n\right) }\beta ^{r}\left( \frac{\beta ^{2}}{ab}-c\right) ^{n}\text{ \ \ \ \
\ \ \ \ \ \ \ \ \ \ \ \ \ \ \ \ \ \ \ \ \ \ \ \ \ \ \ \ \ \ \ \ \ \ \ \ \ \
\ }
\end{equation*}%
\begin{equation*}
=\frac{a^{\xi \left( r+1\right) }}{\left( ab\right) ^{\left\lfloor \frac{r}{2%
}\right\rfloor }}\left( A\alpha _{\xi \left( n\right) }\alpha ^{n+r}-B\beta
_{\xi \left( n\right) }\beta ^{n+r}\right) =\frac{a^{\xi \left( r+1\right) }%
}{\left( ab\right) ^{\left\lfloor \frac{r}{2}\right\rfloor }}\frac{\left(
ab\right) ^{\left\lfloor \frac{n+r}{2}\right\rfloor }}{a^{\xi \left(
n+r+1\right) }}\mathbb{K}_{w,n+r}\text{ \ \ \ \ \ \ \ \ \ \ \ \ \ \ \ \ \ \
\ \ \ \ \ \ \ \ \ \ \ \ \ \ \ \ \ \ \ \ \ \ }
\end{equation*}%
\begin{equation*}
=\frac{a^{-\xi \left( r\right) +\xi \left( n+r\right) }}{\left( ab\right)
^{\left\lfloor \frac{r}{2}\right\rfloor -\left\lfloor \frac{n+r}{2}%
\right\rfloor }}\mathbb{K}_{w,n+r}=\left( ab\right) ^{\frac{n}{2}}\left( 
\frac{a}{b}\right) ^{\frac{-\xi \left( r\right) +\xi \left( n+r\right) }{2}}%
\mathbb{K}_{w,n+r.}\text{ \ \ \ \ \ \ \ \ \ \ \ \ \ \ \ \ \ \ \ \ \ \ \ \ \
\ \ \ \ \ \ \ \ \ \ \ \ \ \ \ \ \ \ \ \ \ \ \ \ \ \ \ \ \ \ \ \ \ \ \ \ \ \ }
\end{equation*}

$\left( ii\right) $ It can be proven similarly.
\end{proof}

\section{Some relations between generalized bi-periodic Fibonacci and Lucas
hybrid numbers}

Now\ we state some relations between bi-periodic Fibonacci and bi-periodic
Lucas hybrid numbers. To do this, we consider the generalized bi-periodic
Fibonacci hybrid numbers $\mathbb{K}_{u,n}$ and the generalized bi-periodic
Lucas hybrid numbers $\mathbb{K}_{v,n}$ which are stated in Table 1.

The Binet formula of $\mathbb{K}_{u,n}$ is%
\begin{equation}
\mathbb{K}_{u,n}=\frac{a^{\xi \left( n+1\right) }}{\left( ab\right)
^{\left\lfloor \frac{n}{2}\right\rfloor }}\left( \frac{\alpha _{\xi \left(
n\right) }\alpha ^{n}-\beta _{\xi \left( n\right) }\beta ^{n}}{\alpha -\beta 
}\right) ,  \label{13}
\end{equation}%
and the Binet formula of $\mathbb{K}_{v,n}$ is 
\begin{equation}
\mathbb{K}_{v,n}=\frac{a^{-\xi \left( n\right) }}{\left( ab\right)
^{\left\lfloor \frac{n}{2}\right\rfloor }}\left( \alpha _{\xi \left(
n\right) }\alpha ^{n}+\beta _{\xi \left( n\right) }\beta ^{n}\right) .
\label{14}
\end{equation}

\begin{theorem}
For any natural number $m,n$ with $n>m,$ we have%
\begin{equation*}
\left( i\right) \mathbb{K}_{u,n+1}+c\mathbb{K}_{u,n-1}=\left( \frac{a}{b}%
\right) ^{\xi \left( n\right) }\mathbb{K}_{v,n}\text{ \ \ \ \ \ \ \ \ \ \ \
\ \ \ \ \ \ \ \ \ \ \ \ \ \ \ \ \ \ \ \ \ \ \ \ \ \ \ \ \ \ \ \ \ \ \ \ \ \
\ \ \ \ \ \ \ \ \ \ \ \ \ \ \ \ \ \ \ \ \ \ \ \ \ \ \ \ \ \ \ \ \ \ \ \ \ \
\ \ \ \ \ \ \ \ \ \ \ \ \ \ }
\end{equation*}%
\begin{equation*}
\left( ii\right) \mathbb{K}_{v,n+1}+c\mathbb{K}_{v,n-1}=\left( \frac{a}{b}%
\right) ^{\xi \left( n\right) }\Delta ^{2}\mathbb{K}_{u,n}\text{ \ \ \ \ \ \
\ \ \ \ \ \ \ \ \ \ \ \ \ \ \ \ \ \ \ \ \ \ \ \ \ \ \ \ \ \ \ \ \ \ \ \ \ \
\ \ \ \ \ \ \ \ \ \ \ \ \ \ \ \ \ \ \ \ \ \ \ \ \ \ \ \ \ \ \ \ \ \ \ \ \ \
\ \ \ \ \ \ \ \ \ \ \ \ \ \ }
\end{equation*}%
\begin{equation*}
\left( iii\right) \mathbb{K}_{u,n}\mathbb{K}_{v,m}-\mathbb{K}_{u,m}\mathbb{K}%
_{v,n}=\left\{ 
\begin{array}{cc}
2\left( -c\right) ^{m}u_{n-m}\left( \mathbb{K}_{v,0}-\theta \right) , & 
\text{ if }n\text{ is even} \\ 
2\left( \frac{a}{b}\right) ^{-\xi \left( m\right) }\left( -c\right)
^{m}u_{n-m}\left( \mathbb{K}_{\widehat{v},0}-\widehat{\theta }\right) , & 
\text{if }n\text{ is odd.}%
\end{array}%
\right. \text{ \ \ \ \ \ \ \ \ \ \ \ \ \ \ \ \ \ \ \ \ \ \ \ \ \ \ \ \ \ \ \
\ }
\end{equation*}
\begin{equation*}
\left( iv\right) \mathbb{K}_{v,n}^{2}-\mathbb{K}_{u,n}^{2}=\left\{ 
\begin{array}{cc}
\begin{array}{c}
\left( \frac{\Delta ^{2}-a^{2}}{\Delta ^{2}}\right) \left( \mathbb{K}%
_{v,0}+\mu _{e}\right) v_{2n}+\left( \frac{\Delta ^{2}-a^{2}}{a^{2}}\right)
\left( \mathbb{K}_{u,0}+\gamma _{e}\right) u_{2n} \\ 
+2\left( -c\right) ^{n}\left( \frac{\Delta ^{2}+a^{2}}{\Delta ^{2}}\right)
\left( \mathbb{K}_{v,0}-\theta \right) ,%
\end{array}
& \text{ if }n\text{ is even} \\ 
\begin{array}{c}
\left( \frac{\Delta ^{2}-a^{2}}{\Delta ^{2}}\right) \left( \mathbb{K}_{%
\widehat{v},0}+\mu _{o}\right) \frac{bv_{2n}}{a}+\left( \frac{\Delta
^{2}-a^{2}}{a^{2}}\right) \left( \mathbb{K}_{\widehat{u},0}+\gamma
_{o}\right) u_{2n} \\ 
+\frac{2b}{a}\left( -c\right) ^{n}\left( \frac{\Delta ^{2}+a^{2}}{\Delta ^{2}%
}\right) \left( \mathbb{K}_{\widehat{v},0}-\widehat{\theta }\right) ,%
\end{array}
& \text{if }n\text{ is odd.}%
\end{array}%
\right. \text{ \ \ \ \ \ }
\end{equation*}
\end{theorem}

\begin{proof}
$\left( i\right) $ From the relations (\ref{13}) and (\ref{14}), we have,%
\begin{equation*}
\mathbb{K}_{u,n+1}+c\mathbb{K}_{u,n-1}\text{ \ \ \ \ \ \ \ \ \ \ \ \ \ \ \ \
\ \ \ \ \ \ \ \ \ \ \ \ \ \ \ \ \ \ \ \ \ \ \ \ \ \ \ \ \ \ \ \ \ \ \ \ \ \
\ \ \ \ \ \ \ \ \ \ \ \ \ \ \ \ \ \ \ \ \ \ \ \ \ \ \ \ \ \ \ \ \ \ \ \ \ \
\ \ \ \ \ \ \ \ \ \ \ \ }
\end{equation*}%
\begin{equation*}
=\frac{a^{\xi \left( n\right) }}{\left( ab\right) ^{\left\lfloor \frac{n+1}{2%
}\right\rfloor }}\left( \frac{\alpha _{\xi \left( n\right) }\alpha
^{n+1}-\beta _{\xi \left( n\right) }\beta ^{n+1}}{\alpha -\beta }\right) +c%
\frac{a^{\xi \left( n\right) }}{\left( ab\right) ^{\left\lfloor \frac{n-1}{2}%
\right\rfloor }}\left( \frac{\alpha _{\xi \left( n\right) }\alpha
^{n-1}-\beta _{\xi \left( n\right) }\beta ^{n-1}}{\alpha -\beta }\right)
\end{equation*}%
\begin{equation*}
=\frac{a^{\xi \left( n\right) }}{\left( ab\right) ^{\left\lfloor \frac{n}{2}%
\right\rfloor +\xi \left( n\right) }}\left( \frac{\alpha _{\xi \left(
n\right) }\alpha ^{n+1}-\beta _{\xi \left( n\right) }\beta ^{n+1}-\alpha
_{\xi \left( n\right) }\alpha ^{n}\beta +\beta _{\xi \left( n\right) }\beta
^{n}\alpha }{\alpha -\beta }\right) \text{ \ \ \ \ \ \ \ \ \ \ \ \ \ \ \ \ \
\ \ \ \ \ \ \ }
\end{equation*}%
\begin{equation*}
=\frac{a^{\xi \left( n\right) }}{\left( ab\right) ^{\left\lfloor \frac{n}{2}%
\right\rfloor +\xi \left( n\right) }}\left( \frac{\alpha _{\xi \left(
n\right) }\alpha ^{n}\left( \alpha -\beta \right) +\beta _{\xi \left(
n\right) }\beta ^{n}\left( \alpha -\beta \right) }{\alpha -\beta }\right) 
\text{ \ \ \ \ \ \ \ \ \ \ \ \ \ \ \ \ \ \ \ \ \ \ \ \ \ \ \ \ \ \ \ \ \ \ \
\ \ \ \ \ \ \ }
\end{equation*}%
\begin{equation*}
=\frac{a^{\xi \left( n\right) }}{\left( ab\right) ^{\left\lfloor \frac{n}{2}%
\right\rfloor +\xi \left( n\right) }}\left( \alpha _{\xi \left( n\right)
}\alpha ^{n}+\beta _{\xi \left( n\right) }\beta ^{n}\right) =\left( \frac{a}{%
b}\right) ^{\xi \left( n\right) }\mathbb{K}_{v,n}\text{ \ \ \ \ \ \ \ \ \ \
\ \ \ \ \ \ \ \ \ \ \ \ \ \ \ \ \ \ \ \ \ \ \ \ \ \ \ \ \ \ \ \ \ \ \ }
\end{equation*}

$\left( ii\right) $ The proof can be done similarly as $\left( i\right) $ by
using the relations (\ref{13}) and (\ref{14}).

$\left( iii\right) $ By using the Binet formulas for $\mathbb{K}_{u,n}$ and $%
\mathbb{K}_{v,n}$, and considering the relation (\ref{7}), we get the
desired result.

$\left( iv\right) $ By using the Binet formulas for $\mathbb{K}_{u,n}$ and $%
\mathbb{K}_{v,n}$, we have%
\begin{equation*}
\Delta ^{2}\left( \mathbb{K}_{v,n}\mathbb{K}_{v,n}-\mathbb{K}_{u,n}\mathbb{K}%
_{u,n}\right) \text{ \ \ \ \ \ \ \ \ \ \ \ \ \ \ \ \ \ \ \ \ \ \ \ \ \ \ \ \
\ \ \ \ \ \ \ \ \ \ \ \ \ \ \ \ \ \ \ \ \ \ \ \ \ \ \ \ \ \ \ \ \ \ \ \ \ \
\ \ \ \ \ \ \ \ \ \ \ \ \ \ \ \ \ \ \ \ \ \ \ \ \ \ \ \ \ \ \ \ \ \ \ \ \ \
\ \ \ \ \ \ \ \ \ \ }
\end{equation*}%
\begin{eqnarray*}
&=&\frac{a^{2\xi \left( n+1\right) }}{\left( ab\right) ^{2\left\lfloor \frac{%
n}{2}\right\rfloor }} \\
&&\frac{\Delta ^{2}}{a^{2}}\left( \alpha _{\xi \left( n\right) }\alpha _{\xi
\left( n\right) }\alpha ^{2n}+\beta _{\xi \left( n\right) }\beta _{\xi
\left( n\right) }\beta ^{2n}+\left( \alpha \beta \right) ^{n}\left( \alpha
_{\xi \left( n\right) }\beta _{\xi \left( n\right) }+\beta _{\xi \left(
n\right) }\alpha _{\xi \left( n\right) }\right) \right)  \\
&&-\left( \alpha _{\xi \left( n\right) }\alpha _{\xi \left( n\right) }\alpha
^{2n}+\beta _{\xi \left( n\right) }\beta _{\xi \left( n\right) }\beta
^{2n}-\left( \alpha \beta \right) ^{n}\left( \alpha _{\xi \left( n\right)
}\beta _{\xi \left( n\right) }+\beta _{\xi \left( n\right) }\alpha _{\xi
\left( n\right) }\right) \right) \text{ \ \ \ \ \ \ \ \ \ \ \ \ \ \ \ \ \ \
\ \ \ \ \ \ \ \ \ \ \ \ \ \ \ \ \ \ \ \ \ \ \ \ \ }
\end{eqnarray*}%
\begin{eqnarray*}
&=&\frac{a^{2\xi \left( n+1\right) }}{\left( ab\right) ^{2\left\lfloor \frac{%
n}{2}\right\rfloor }}\left( \left( \frac{\Delta ^{2}}{a^{2}}-1\right) \left(
\alpha _{\xi \left( n\right) }\alpha _{\xi \left( n\right) }\alpha
^{2n}+\beta _{\xi \left( n\right) }\beta _{\xi \left( n\right) }\beta
^{2n}\right) \right.  \\
&&\left. +\left( \frac{\Delta ^{2}}{a^{2}}+1\right) \left( \alpha \beta
\right) ^{n}\left( \alpha _{\xi \left( n\right) }\beta _{\xi \left( n\right)
}+\beta _{\xi \left( n\right) }\alpha _{\xi \left( n\right) }\right) \right)
.\text{ \ \ \ \ \ \ \ \ \ \ \ \ \ \ \ \ \ \ \ \ \ \ \ \ \ \ \ \ \ \ \ \ \ \
\ \ \ \ \ \ \ \ \ \ \ \ \ \ \ \ \ \ \ \ \ \ \ \ \ \ \ \ \ \ \ \ \ \ \ \ \ \
\ \ }
\end{eqnarray*}%
By considering the relations (\ref{7}), (\ref{9}), and (\ref{10}), we get
the desired result.
\end{proof}

\section{Conclusion}

In recent years, many studies have been devoted to investigate the hybrid
numbers whose components are from the special number sequences such as
Fibonacci, Lucas, Pell, Horadam numbers, etc. This work provides a new
generalization for hybrid numbers whose coefficients are from the Horadam
numbers. Most of the results of this study generalize the results of those
were given in \cite{szynal3, morales, senturk}. It would also be interesting
to study the algebraic structure of these new hybrid numbers.

\end{document}